\documentclass[a4paper,12pt]{amsart}

\usepackage{latexsym}
\usepackage{amsmath, amssymb, graphics, setspace}
\usepackage{amsfonts}
\usepackage{mathtools}
\usepackage{setspace} 
\usepackage{capt-of}
\usepackage{natbib}
\usepackage{booktabs}  
\usepackage{hyperref}
\usepackage{array}
\usepackage{leftidx}
\usepackage{psfrag}
\usepackage{color}

\usepackage{marginnote} 
   
\newcommand\preq{\stackrel{\scriptsize \mathclap{\mbox{d}}}{=}}

\newcommand\prsim{\stackrel{\scriptsize \mathclap{\mbox{d}}}{\sim}}
\newcommand\mn[1]{
}

\newcommand\omn[1]{
}

\newtheorem{Lemma}{Lemma} 
\newtheorem{Theorem}{Theorem} 
 
\newtheorem{Corollary}{Corollary}

\theoremstyle{definition}

\theoremstyle{remark} 
\newtheorem{Remark}{Remark} 
\newtheorem{Example}{Example}

\newcommand{\E}{\mathbb{E}}

\begin{document}

\title{On Sibuya-like  distributions in branching and birth-and-death processes}
\author[L.B. Klebanov]{Lev B. Klebanov}
\address{Department of Probability and Mathematical Statistics\\ 
Faculty of Mathematics and Physics of the Charles University\\
Sokolovská 49/83, 186 75 Prague, Czech Republic}
\email{Lev.Klebanov@mff.cuni.cz}
\author[M. \v{S}umbera]{Michal \v{S}umbera}
\address{Nuclear Physics Institute\\ Czech Academy of Sciences\\ 25068 \v{R}e\v{z}, Czech Republic}
\email{sumbera@ujf.cas.cz}

 	\date{\today}
\maketitle

\begin{abstract}
We report some properties of heavy-tailed Sibuya-like distributions related to thinning, self-decomposability and  branching processes. 
Extension of the thinning operation of on-negative integer-valued random variables to scaling by arbitrary positive number leads to a new class of probability distributions with generating function $Q(w)$ expressible as a Laplace transform $\varphi(1-w)$ and  probability mass function $p_n$ satisfying  simple one step recurrence relation  between $p_{n+1}$ and $p_n$. We show that the compound Poisson-Sibuya and the shifted Sibuya distributions belong to this class. 
Using the fact that the same Markov property is present in stationary solutions of the birth and death equations we identify the Sibuya distribution and some of its variants as particular solutions of these equations. 
We also establish condition when integer-valued non-negative heavy-tailed random variable has finite $r$-th absolute moment ($0 < r < a < 1$).

\end{abstract}

\date{today}

\maketitle

\section{Self-decomposability and scaling of non-negative integer-valued random variables}
Infinitely divisible  distributions  \cite{Feller1,JKK2005,StvH2004} play nowadays an important role in several parts of the probability theory. They are also frequently used by the models based on the sum of several independent quantities with the same distribution because its infinite divisibility appears to be a convenient assumption \cite{StvH2004}. In the simplest case of the non-negative integer random variable (rv) $Y\in \mathbb{Z}_+$ the condition of  infinite divisibility means that  $\forall n\in\mathbb{N}$ it can be written as a random sum $Y_n=X_1+X_2 \ldots +X_n$ of independent identically distributed (iid)  rvs $X_i\in \mathbb{Z}_{+}, i=1,n$. Consequently, if $\mathbb P(N=n)=p_n$  is the probability mass function (pmf) and $Q(w)=\sum_n p_n w^n$ the probability generating function (pgf) of the rv $N$ and $G(w)=\mathbb{E}[w^{X_i}]$ is the pgf of the rvs $X_i$ then the $n$th root of the compound pgf 
\begin{equation}
\label{eq:comp}
H(w)=G(Q(w)):= G\circ Q
\end{equation}
is $\forall n$ also the pgf. This is equivalent to saying that $H(w)$ can be expressed as the pgf of compound Poisson distribution \cite{Feller1}: 
\begin{equation}
\label{eq:cpd}
H(w)=G\circ Q=P\circ R\,,~~ P(\lambda,w)=e^{-\lambda(1-w)}\,,~
R(w)=1+\frac{\log H(w)}{\lambda}
\end{equation}
where $P(w)$ is the pgf of the Poisson distribution and $R(w)$ is the pgf of positive discrete rv with $R(0)=0$.

Important class of the infinitely divisible distributions on $\mathbb{Z}_+$ are those which are self-decomposable \cite{StvH2004}. Let us recall that discrete random variable $X\in \mathbb{Z}_{+}$ with the pgf $G(w)$  is called self-decomposable if  it can be written as a sum of two independent variables \cite{StvH2004}
\begin{equation}
\label{seldec}
  X = a \odot X + Y_{a};\;\;\; \forall a \in (0,1). 
\end{equation} 
 The thinning operator $\odot$ 	\cite{Renyi1956, StvH79} is defined in the following way
	\begin{equation}
		a \odot X:=   \sum_{i=1}^X Z_i,  ~Z_i \preq \mathcal{B}_{a}, ~ 
		G_{a \odot X}(w)\!:=\! G(\mathcal{B}_{a}(w))\!=\!G(1\!- \! a (1\! -\! w)) \label{eq:thin}
	\end{equation}
where $\preq$ means equality in distribution,  $\mathcal{B}_{a}$ is two-valued Bernoulli-distributed random variable with $p_0=1-a$ and $p_1=a$ and $\mathcal{B}_{a}(w)$ is its pgf.  In terms of pgfs  Eq.~(\ref{seldec})  reads:
\begin{equation}
\label{eq:seldpgf}
G(w)= G(\mathcal{B}_{a}(w))H_a(w)\, ,    
\end{equation} 
where $H_a(w)$ is the pgf of the rv $Y_a$. It can be shown that the function $H_a(w)$ is absolutely monotone and 
infinitely divisible \cite{StvH2004}. Obviously if $G_1(w)$ and $G_2(w)$ are self-decomposable pgfs then their product $G_1(w)G_2(w)$ is also self-decomposable pgf.
More general definitions of thinning operators were given in \cite{KlSl, Kl2021}.

The importance of the thinning operation (\ref{eq:thin}) for the rvs with finite mean is explained by the following theorem. 
\begin{Theorem}
\label{eq:binrand}
Let  $Q(w)$ be a pgf of the rv $X$ such that $\E X=Q^{(1)}(1)<\infty$ and with scaled factorial moments $F_j$ of order $j$
\begin{equation}
\label{eq:Fj}
F_j := \frac{\E X^{(j)}}{(\E X)^j}=\frac{\E X(X-1)\ldots (X-j+1)}{(\E X)^j} = \frac{Q^{(j)}(1)}{(Q^{(1)}(1))^j}, ~j\!=\!1,2,3,\ldots .
\end{equation}
Then the rv $Y=a\odot X$ with the pgf $G(w)=Q_{a \odot X}(w)$ has identical scaled factorial moments as $X$.
 \end{Theorem}

\begin{proof}
\begin{equation}
  G^{(j)}(w)=a^j Q^{(j)}(1 - a + a w); \;\;\;
 \frac{G^{(j)}(1)}{(G^{(1)}(1))^j} = \frac{a^jQ^{(j)}(1)}{(a Q^{(1)}(1))^j} = F_j
\end{equation}
\end{proof}

\begin{Example}
\label{ex:nbdfact}
Consider rv $X$ with $\langle n \rangle:=\E X< \infty$  which has the negative binomial distribution (NBD) with the pgf
\begin{equation} 
\label{eq:nbdpgf}
Q_{NBD}(w) =  \Bigl (1 + \frac{\langle n \rangle}{k}(1-w) \Bigr)^{-k}; \;\; k>0.
\end{equation}
Its factorial moments $F_j= (k+j-1)_j/k^j$, where $(x)_n=\prod_{\ell=0}^{n-1}(x-\ell)$ is the falling factorial, do not depend on $\langle n \rangle$ and therefore its thinned version $Y=\odot X$  with the pgf
\begin{equation}
Q_{NBD}(\mathcal{B}_{a}(w))\!=\!  \Bigl( 1\!+\! \frac{\langle n \rangle}{k} \Bigl (1\!-\!\mathcal{B}_{a}(w) \Bigr)  \Bigr)^{-k} \! = \! \Bigl(1\!+\! \frac{ a\langle n \rangle}{k} (1\!-\!w)  \Bigr)^{-k}
\label{eq:nbdthn}
\end{equation}
has not only the same scaled factorial moments as the rv $X$ but is also form-invariant w.r.t. scaling  $\langle n \rangle \to a \langle n \rangle$.
\end{Example}

Returning back to the self-decomposable distributions we note, that there is only one distribution, discrete stable distribution (DSD), for which the infinitely divisible component $Y_a$ is also self-decomposable  
\begin{equation}
\label{eq:sfd}
X = a \odot X + Y_{a} = a \odot X + b \odot X\, ,~~ a^{\delta}+b^{\delta}=1, ~~\forall \delta \in(0,1)\, .
\end{equation}
 The  canonical form (\ref{eq:cpd}) of its pgf reads
\begin{eqnarray}
\label{eq:dsdpgf} \;\;
Q_{DSD}(\lambda,\gamma,w) = P\circ \mathcal{S}= e^{-\lambda(1-w)^{\gamma}};\;\;
\mathcal{S}(\gamma,w)=1-(1-w)^{\gamma},\;\;\gamma \in (0,1)\, ,
\label{eq:sbdpgf}
\end{eqnarray}
where $P(\lambda, w)$ and $\mathcal{S}(\gamma,w)$ are the pgfs of Poisson and Sibuya distributions \cite{art:Sibuya, Dev1993, Huillet2018}, respectively. 
	
Let us note that there is a wide class of self-decomposable pgfs which similarly to  Poisson (\ref{eq:cpd}), Bernouily (\ref{eq:seldpgf}), negative binomial (\ref{eq:nbdpgf}) and discrete stable distributions (\ref{eq:dsdpgf}) depend explicitly on the argument $1-w$, i.e. allow expansion of the type
\begin{equation}
 Q(w)=\sum_{n\geq 0} p_n w^n= 1+\sum_{n>0} b_n (1-w)^n \, , \;\;\;  \sum_{n>0} b_n > -1\, .
\end{equation}
 This category comprises also some less frequently used ones like the Mittag-Leffler \cite{Huillet2018} or Discrete Linnik   \cite{Dev1993,Huillet2018} distributions. 
 Rather general construction of such pgfs is given in \cite{Kl2021}. More precisely, let $\varphi (s)$ be the Laplace transform of a positive random variable $X$. Then 	
\begin{equation}
\label{eq:k1}
	Q(w)=\varphi (1-w) 
\end{equation}
is a probability generating function. 

The change of variable $w = (1-a) +a z$ in (\ref{eq:k1}) gives us 
\begin{equation}
\label{eq:k2}
	Q(1-a+az)=\varphi (a(1-z)) .
\end{equation}
In other words, application of the thinning operator $w \rightarrow 1-a+az$ is equivalent to the scale change from $1$ to $a$ in the case of $a \in (0,1)$. We can apply the thinning operator to arbitrary positive integer random variable but for the case of $0<a<1$. However, $\varphi(as)$ remains to be Laplace transform for any positive $a$. Naturally, we came to the problem: {\it Describe all probability generating functions $Q(w)$ such that $Q(1-a+aw)$ is again the pgf for all $a>0$}. The solution is given by the following result.
\begin{Theorem}\label{thK}
Let $Q(w)$ be a probability generating function. $Q(1-a+aw)$ is the pgf for all $a>0$ if and only if there exists a Laplace transform of a positive random variable $\varphi (s)$ such that the representation (\ref{eq:k1}) holds. 
\end{Theorem} 
\begin{proof}
Without loss of generality we may consider the case $a>1$.

1. Suppose that $\varphi(s)$ is Laplace transform of a distribution function and
$Q(w)=\varphi(1-w)$ is probability generating function. However, $\varphi(as)$ with arbitrary $a>0$ is Laplace transform as well. According to the result from \cite{Kl2021} mentioned above, $\varphi(a(1-w))=Q(1-a+aw)$ is probability generating function for any $a>0$.

2. Suppose that $Q(1-a+aw)$ is probability generating function for any $a>0$.
Define $\varphi(s)=Q(1-s)$. It is necessary to proof that $\varphi(s)$ is Laplace transform of a distribution function or, equivalently, that it is absolutely monotone function. In other words we have to proof that
\begin{equation}
\label{eq:k3}
\varphi^{(k)}(s)=(-1)^k A_k(s), \quad k=0,1,\ldots , 
\end{equation}  
where the functions $A_k(s)$ are non-negative for $s>0$. For any $w\in (0,1)$ and any $a>0$ we have $\varphi(a(1-w)) = Q(1-a+aw)$. Therefore,
\begin{equation}\label{eq:k4}
\frac{d^{k}}{dw^{k}}\varphi(a(1-w))=(-1)^{k}a^k\varphi^{(k)}(a(1-w))=(-1)^{k}a^kQ^{(k)}(1-a+aw).
\end{equation} 
Because $Q(1-a+aw)$ is probability generating function for any $a>0$ the terms $a^kQ^{(k)}(1-a+aw)$ are non-negative for all $0<w<1$ and all positive $a$. It proves absolutely monotones of $\varphi(s)$. 
\end{proof}

\begin{Corollary}
\label{cor:g}
Let $\mathbb P(N=n)=p_n=Q^{(n)}(0)/n!\,, n\in\mathbb{Z}_{+}$ be a pmf of random variable $N$ with the pgf $Q(w)$ such that 
\begin{equation}
\label{eq:lapdens}   
Q(w)=\varphi (1-w) = \int_0^\infty e^{-(1-w)x}f(x) dx 
\end{equation}
where $\varphi(s)$ is the Laplace transform of a positive random variable $X$ with the  probability density $f(x)=dF_{X}(x)/dx,~F_{X}(x)=\mathbb P(X<x)$.  Then $p_n$ satisfies the one-step recurrence relation \begin{equation}
\label{eq:gnkl}
p_{n+1}=\frac{p_n}{n+1}g(n);\;\;\;\; 
g(n) = \frac{Q^{(n+1)}(0)}{Q^{(n)}(0)}= \frac{(n+1)p_{n+1}}{p_n}=\frac{\int_{0}^{\infty}e^{-x}x^{n+1}f(x) dx}{\int_{0}^{\infty}e^{-x}x^{n} f(x) dx}.
\end{equation}
\end{Corollary}

Let us note that if pgf $Q(w)$ has the representation (\ref{eq:lapdens}) then for any $b \in (0,1]$ the function $Q(bw)/Q(b)$ is pgf and has the same representation with corresponding function $f_b(x)$. Namely,
\begin{equation}
\label{eq:lapdensa} 
 f_b(x) =e^{-(1-b)x/b}f(x/b)/Z(b)\,, ~~Z(b)= 
b\int_{0}^{\infty}e^{-(1-b)u}f(u)du. 
\end{equation}

It is also worth mentioning that the rv $Y\in \mathbb{Z}_+$ with the pgf $Q(w)$ satisfying Eq.~(\ref{eq:lapdens}) 
can be also interpreted as Poisson--distributed random variable  whose mean $\mathbb{E} Y=X$ fluctuates according to the probability density $f(x)$ or, equivalently, as a mixture of the Poisson pgfs $P(x,w)=e^{-x(1-w)}$ with the mixture weight $f(x)$. 
Let us give few examples.

\begin{Example}
\label{ex:poistr}
Consider pgf of the Hermite distribution \cite{JKK2005} $Q(w)=e^{-\lambda_1(1-w)}e^{-\lambda_2(1-w)^2}$ describing sum of two independent random variables $X_1$, $X_2$ having pgf $e^{-\lambda_1(1-w)}$ and $e^{-\lambda_2(1-w)^2}$ correspondingly. Values of $X_1$ are concentrated on $n=0,1,2,3\ldots$, of $X_2$ on even values only $n=0,2,4,6\ldots$. Hermite distribution is infinitely divisible. Its pmf satisfies recursion relation $(n\!+\!1)p_{n+1} \!=\! \lambda_1 p_n \!+ \! 2\lambda_2 p_{n-1}$. For $\lambda_2=0$ the recursion simplifies
to $(n\!+\!1)p_{n+1} \!=\! \lambda_1 p_n$ and  pgf $Q(w)$ can be obtained  from  Eq.~(\ref{eq:k1}) with $f(x)=\delta(x-\lambda_1)$. However, for $\lambda_2\neq0$ the density $f(x)$ does not exist.
On the opposite side is the notoriously known example of the negative binomial distribution with the pgf (\ref{eq:nbdpgf})
which can be obtained from Eq.~(\ref{eq:k1}) taking for $f(x)$ the gamma distribution density $f(x)=(\beta^{k}/\Gamma(k))x^{k-1}e^{-\beta x}$ with $\beta=(1-q)/q$. 
\end{Example}

\begin{Example}
\label{ex:dsd}
Let's compare  the pgf (\ref{eq:dsdpgf}) of discrete stable distribution with the Laplace integral \cite{Pollard1946-jy}
\begin{equation}
\label{eq:Pollard}
e^{-r^{\alpha}} = \int_0^\infty \!e^{-rx}g_{\alpha}(x)dx\,, ~~
g_{\alpha}(x)= \frac{1}{\pi} \sum^{\infty}_{j=1}\frac{(-1)^{j+1}}{j!x^{1+\alpha j}}\Gamma(1\!+\!\alpha j)\sin(\pi\alpha j)\, ,
\end{equation}
where $g_{\alpha}(x)$ is the probability density function of the one-sided continuous stable distribution. Exact and explicit expressions for $g_{\alpha}(x)$, for all $\alpha=l/k<1$, with $k$ and $l$ positive integers can be found in \cite{PenGo}. The substitutions $\gamma=\alpha$ and $r=\lambda^{-\gamma}(1-w)$  in  Eq.~(\ref{eq:Pollard}) enable us to express the pgf (\ref{eq:dsdpgf}) as the Laplace transform  (\ref{eq:lapdens}) of the function 
\begin{equation}
\label{eq:fdsd}
f(x)=\lambda^{\gamma}g_{\gamma}(\lambda^{\gamma}x).
\end{equation}
Hence the pmf of the discrete stable distribution satisfies the one step recurrence relation $(n+1)p_{n+1}=g(n) p_n$ with $g(n)$ given by Eq.~(\ref{eq:gnkl}) and $f(x)$ by Eq.~(\ref{eq:fdsd}). It is worth mentioning that since a closed-form expression for $p_n$ using elementary functions  is unknown the proof of this relation by the other methods is almost impossible.
\end{Example}

\begin{Example}
\label{ex:shsbd}
Consider random variable $Y\preq X+1$ where $Y\in \mathbb{N}$ has Sibuya distribution.
Then the rv $X \in \mathbb{Z}_{+}$ has shifted Sibuya distribution which is self-decomposable \cite{Christoph2000} and therefore also infinitely divisible \cite{StvH2004}.  Its pgf can be expressed via Eq.~ (\ref{eq:lapdens}) 
\begin{eqnarray}
\label{eq:sh_sbdpgf}
\mathcal{S}_0(\gamma,w)=\frac{\mathcal{S}(\gamma,w)}{w} = \frac{1-(1-w)^{\gamma}}{w} =
\int_0^\infty e^{-(1-w)x}f(x) dx; \\
 f(x)=-\frac{e^x \Gamma (-\gamma,x)}{\Gamma (-\gamma )}
\approx x^{-1-\gamma}\left(\frac{\gamma}{\Gamma(1-\gamma)}+\mathcal{O}\left(\frac{1}{x}\right)\right)\,,
\label{eq:sh_sbdfx}
\end{eqnarray}
where $\Gamma (a,x)$ is the incomplete gamma function with $\Gamma (a,0)=\Gamma (a)$. From Eq.~ (\ref{eq:gnkl}) we obtain
\begin{equation}
\label{g_shsbd}
g(n) = \frac{\int_{0}^{\infty}x^{n+1}\Gamma (-\gamma ,x) dx}{\int_{0}^{\infty}x^{n} \Gamma (-\gamma ,x) dx}
= \frac{(n+1)(n+1-\gamma)}{n+2}.
\end{equation}
\end{Example}
It is worth mentioning that the $r$-th absolute moment  of the probability density function $f(x)$ (\ref{eq:sh_sbdfx}) 
\begin{equation}
\int_0^{\infty}x^r f(x) dx = \frac{\sin (\pi  \gamma ) \Gamma (r+1)}{\sin (\pi  (\gamma - r ))} 
\end{equation}
diverges for $r>\gamma$. This well-known property of the continuous heavy-tailed  distributions is also shared by the family of discrete distribution represented by the Sibuya distribution. In the latter case the mean and higher moments $\langle n^r \rangle =Q^{(r)}(w=1),~ r \geq 1$ do not exist.  One may  na\"{i}vely expect that in this case also the Shannon entropy $S=-\sum_n p_n \log p_n$ of these distributions diverges. However, this is not the case. The entropy $S$ is finite either if $\langle \log n \rangle = \sum_{n=1}^{\infty}p_n \log n < \infty $ or if $\exists ~r> 0$ such that $\langle n^{r} \rangle = \sum_{n=1}^{\infty}p_n n^{r}< \infty$ \cite{Bac2013}.
The following theorem establishes the condition when  $\langle n^{r} \rangle< \infty, ~0<r<1$.
\begin{Theorem}
Integer-valued non-negative random variable $Z$ with probability generating function $Q(w)$ has finite $r$-th absolute moment $(0<r<a<1)$ iff 
 \begin{equation}
 \label{momint}
 \int_{0}^{1}\frac{1-Q(w)}{(-\log{w})^{1+r}w}dw < \infty
 \end{equation}
\label{momr}
\end{Theorem}
\begin{proof}
For continuous non-negative random variable $X$ with Laplace transform $\varphi(u)=\E e^{-uX}$ its $r$-th absolute moment $\E X^r, 0<r<1$ can be calculated using identity 
\begin{eqnarray}
\int_{0}^{\infty}\frac{1-e^{-Xu}}{u^{1+r}}du = X^r \int_{0}^{\infty}\frac{1-e^{-z}}{z^{1+r}}dz =X^{r}(-\Gamma(-r)); ~ z=Xu, ~ 0<r<1   \\
 \E X^r = \frac{1}{-\Gamma(-r)}\int_{0}^{\infty}\frac{1-\varphi(u)}{u^{1+r}}du.
 \end{eqnarray}
 Similarly, for the {\it integer-valued} non-negative random variable $Z$ with Laplace transform $\varphi(u)=Q(e^{-u})$ we have
\begin{equation} 
 \E Z^{r} =\frac{1}{-\Gamma(-r)} \int_{0}^{\infty}\frac{1-Q(e^{-u})}{u^{1+r}}du = \frac{1}{-\Gamma(-r)} \int_{0}^{1}\frac{1-Q(w)}{(-\log{w})^{1+r}w}dw.
 \end{equation}
\end{proof}

\begin{Example}
Consider $N$ extended objects--particles each of the same unit volume $v_0=1$. Particles are incompressible and densely packed occupying the total volume $V=N$ in three dimensional space.  If  $N$ is the Sibuya-distributed rv with the pgf $\mathcal{S}(\gamma,w)$  (\ref{eq:sbdpgf}) then  the integral (\ref{momint}) and hence the corresponding moments $\E N^r$ are finite only for $0<r<\gamma<1$. Consequently the  mean $\E V$ of the fluctuating volume $V\preq N$ is nonexistent. However, for $\gamma>1/2$ the mean value of the enclosing  surface $S \prsim N^{1/2}$ as well as of the linear extension $R \prsim N^{1/3}$, where  $\prsim$ means proportionality in distribution, of the 3d--space occupied by particles  are finite. Moreover for $1/2\geq \gamma > 1/3$ $\E S$ does not exist but $\E R<\infty$.
\end{Example}

Similarly to the pgfs of Bernoulli and geometric distributions Sibuya pgfs form a commutative semi-group under the operation of compoundig $\circ$
\begin{equation}
 \mathcal{S}(\gamma_1)\circ \mathcal{S}(\gamma_2)  =
 \mathcal{S}(\gamma_1\gamma_2)
\end{equation}

The following theorem establishes the additional similarity with  Bernoulli distribution.
\begin{Theorem} {\rm \cite{Sapa1995}}
Let $Q(w)$ be the  self-decomposable pgf  of the random variable $X$ on $\mathbb{Z}_{+}$ and  $\mathcal{S}(\gamma, w)$ the pgf of Sibuya distribution. Then $\forall \gamma \in (0,1)$  the function $G(w)= Q\circ \mathcal{S}(\gamma)$ is also a self-decomposible pgf  on $\mathbb{Z}_{+}$.
\label{sfdsbd}
\end{Theorem}
\begin{proof}
We need to prove that $\forall a$ the function 
\begin{equation}
\label{ratG}
\frac{G(w)}{G(\mathcal{B}_{a}(w))}
= \frac{Q(\mathcal{S}(\gamma,w))}{Q(1\!-\!a^{\gamma}(1\!-\!\mathcal{S}(\gamma,w)))}\!=\! 
\frac{Q(\mathcal{S}(\gamma,w))}{Q(\mathcal{B}_{a^{\gamma}}(\mathcal{S}(\gamma,w)))}
\!=\!Q_{a^{\gamma}}(\mathcal{S}(\gamma,w))
\end{equation}
is also the pgf. This is true since $Q_{\alpha^{\gamma}}(\mathcal{S}(w))$ is a compound pgf $\forall a \in(0,1)$.  
\end{proof}

\begin{Remark}
Obviously, sum of $n$ independent self-decomposable rvs $X_1+\ldots X_n$ is also self-decomposable. The non-trivial character of the above theorem consists in the statement that this remains true also for the sum iid self-decompolsable rvs when $n$ fluctuates according to the Sibuya distribution.
\end{Remark}

\begin{Remark}
Note that the thinned version of the Sibuya pgf appearing in Eq.~\ref{ratG}
\begin{equation}
\label{scsbdpgf}
1 - a^{\gamma}(1-\mathcal{S}(\gamma, w))=1-\lambda(1-w)^{\gamma}:=\mathcal{S}(\lambda,\gamma,w);\;\;\; \lambda>0
\end{equation}
is the pgf of the {\it scaled Sibuya distribution} \cite{Christoph2000}. The latter is known to be infinitely divisible if and only if $\lambda\leq 1-\gamma $ and self-decomposable if and only if $\lambda\leq (1-\gamma)/(1+\gamma)$ \cite{Christoph2000}. In this case, as follows from the Theorem \ref{sfdsbd}, the pgf $\mathcal{S}(\lambda,\gamma)\circ \mathcal{S}(\alpha)$ is also self-decomposable.
\end{Remark}

\begin{Corollary}
Every Sibuya-distributed random variable $X_1$ with parameter $\gamma_1$ can be decomposed into the sum of two rvs $X_1=X_2+Y$, where $X_2$ has the Sibuya distribution with parameter $\gamma_2<\gamma_1$ and $Y$ is self-decomposable.
\end{Corollary}
\begin{proof} Let  $Q(w)$ be the pgf of the rv $Y$ and $\gamma_1=\alpha\gamma_2, 0<\alpha<1$. Then 
\begin{equation}
  \mathcal{S}(\alpha \gamma_2,w) =
 \mathcal{S}(\alpha)\circ\mathcal{S}(\gamma_2)= \mathcal{S}(\gamma_2,w)Q(w). 
\end{equation}
and hence 
\begin{equation}
 Q(w)=\frac{\mathcal{S}(\alpha)\circ\mathcal{S}(\gamma_2)}{\mathcal{S}(\gamma_2,w)}=
 \left(\frac{\mathcal{S}(\alpha,w)}{w}\right)\circ \mathcal{S}(\gamma_2)=
  \mathcal{S}_0(\alpha)\circ \mathcal{S}(\gamma_2)
  \,,   
\end{equation}
where $\mathcal{S}_0(\alpha,w)$ is the pgf of the shifted Sibuya distribution (\ref{eq:sh_sbdpgf}). 
The latter is self-decomposable \cite{Christoph2000} and so we can apply the Theorem \ref{sfdsbd} to obtain the needed result.
\end{proof}

\begin{Example}
\label{ex:mtlef}
Skipping the trivial case of the discrete stable distribution (\ref{eq:dsdpgf}) which is self-decomposible  by construction let us consider the negative binomial distribution with parameter $k=1$, so called  geometric distribution, with the pgf $Q(w)=(1+\lambda(1-w))^{-1}$, cf. Eq.~(\ref{eq:nbdpgf}). Then
\begin{equation}
\label{MLpgf}
 G(w)= Q\circ \mathcal{S}(\gamma)= \frac{1}{1+\lambda(1-w)^{\gamma}}
\end{equation}
is the pgf of Mittag-Leffler distribution. Since $Q(w)$ is self-decomposable it is also $G(w)$. The random variable $Y\preq X+1$ with the pgf $wG(w)$ where $X$ has the pgf given by Eq.~(\ref{MLpgf}) with $\gamma=\frac{1}2{}$ and $\lambda=1$ describes the number of trials until the first return to the origin in one-dimensional symmetric random walk on  the  lattice \cite{Feller1}.
\end{Example}


\section{Birth-and-death process}
Let us consider continuous-time birth-and-death (B-D) process \cite{Feller1,Anderson2012-cf}  consisting of random events of the same kind -- e.g. creations and disintegrations of particles, populations of interacting species with indistinguishable members, queueing networks (interconnected queues) {\it etc.}. 
The system represents a continuous-time Markov chain with a discrete state space $n\in \mathbb N$ which changes only through transitions from states to their nearest neighbors $n\to n\pm1$. It is described by stochastic process $\{X(t): t \geq 0\}$
\begin{equation}
\mathbb P(X(s+t)=j)|X(s)=i)= \mathbb P(X(t)=j)|X(0)=i)=P_{i,j}(t); \;\; \forall i,j, \;t>0,s>0
\end{equation}
with transition probabilities $P_{i,j}(t)$ satisfying Chapman-Kolmogorov equation
\begin{equation}
    P_{i,j}(s+t)=\sum_{k}P_{i,k}(s)P_{k,j}(t).
\end{equation}
Evolution of the probability distribution $p_j(t)\equiv\mathbb P (X(t)=j)=\sum_i \mathbb P(X(0=i)P_{i,j}(t)$ to find the system in state $j$ at time $t$ if at time $t=0$ it was in state $i$ 
is controlled by the time-independent transition probabilities per unit time 
\begin{equation}
    \lambda_j=\frac{dP_{j,j+1}(t)}{dt}\Big |_{t=0^{+}}\,, \;\;\;\;
    \mu_j=\frac{dP_{j,j-1}(t)}{dt}\Big |_{t=0^{+}}.
\end{equation}
and satisfies  differential-difference equation
\begin{eqnarray}
\label{dd}
\frac{dp_{j}(t)}{dt}=-(\lambda_j +\mu_j) p_{j}(t) + \lambda_{j-1} p_{j-1}(t) +\mu_{j+1}p_{j+1}(t), \;\; j > i\,,
\\ \nonumber 
 \frac{dp_i(t)}{dt}=-\lambda_i  p_i(t) + \mu_{i+1} p_{i+1}(t), \;\; i \in \mathbb N  \,.
\end{eqnarray}
The last equation in (\ref{dd}) reflects the initial condition  $p_j(0)=\delta_{ij}$. 

One of the commonly used models specifying functional dependence of the transitional probabilities $\lambda_j$ and $\mu_j$ on $j$ considers $j\to j\pm 1$ transition as resulting from several underlying microscopic (particle-level) processes $k \to k\pm 1$,   $k=0,\ldots,j$ with the elementary transition probabilities $\alpha_k$ and $\beta_k$, respectively. Thus one particle is born or dies with the elementary probability per unit of time $\alpha_0$ or $\beta_0 j$, one particle can split into two particles with the rate $\alpha_1 j$, two particles can fuse into a single particle with the rate $\beta_1 j(j-1)$, {\it etc.}
The relations between  the transitional probabilities $\lambda_j$ and $\mu_j$ and elementary ones are given by the following formulas
\begin{equation}
\label{albe}
\lambda_j  =\sum_{k=0}^{\ell} \alpha_k (j)_k, \; \;\;\;
\mu_j  =  j \sum_{k=0}^{m} \beta_k (j-1)_k; \; \;\;\; \ell \le j ;\; m \le j ;\;\;\;\;\; (j)_k \!\equiv\! \prod_{i=0}^{k-1} (j\!-\!i)
\end{equation}
For application of this approach to multiparticle production in high energy physics see e.g. \cite{Biyajima:1986yy, Zborovsky:1996yj}.

\subsection{Stationary B-D equations} It is well known that for large $t$ solution of the B-D Eq.~ (\ref{dd}) can be well approximated by  the limiting probabilities $p_j=\lim_{t\to\infty}p_j(t)$ \cite{Feller1,Anderson2012-cf} which solve the Eq.~ (\ref{dd}) with $dp_j(t)/dt=0$.  The latter is in this case transformed into two functional equations for one unknown function $F(j)$ 
\begin{eqnarray}
F(j\!+\!1) \equiv \mu_{j+1} p_{j+1} - \lambda_j p_{j} = \mu_jp_{j}-\lambda_{j-1}p_{j-1} \equiv F(j), ~~j > i\\ \nonumber  
F(i+1)\! = \! \mu_{i+1}p_{i+1}\! -\! \lambda_i p_{i} \!=0, \;\;i \in \mathbb N .
\label{Fn}
\end{eqnarray}
 Solution of the first equation $F(j\!+\!1)\!=\!F(j)$  is fixed to zero by the second equation $F(i+1)=0$. Consequently
\begin{equation}
\frac{p_{j+1}}{p_j}  = \frac{\lambda_j}{\mu_{j+1}} = \frac{\sum_{k=0}^{\ell} \alpha_k (j)_k}{(j+1)\sum_{k=0}^{m} \beta_k (j)_{k}}.
\end{equation}
We have thus proven the following theorem.
\begin{Theorem}
The unique solution of Eqs. (\ref{Fn})  with  $\lambda_j $  and $\mu_j$  given by Eq.~(\ref{albe}) reads:
\begin{equation}
p_j\!=\!p_i\prod_{n=i}^{j-1}\frac{g(n)}{n+1} ; \; \;\;\;\; 
g(n) = \frac{(n\!+\!1)p_{n+1}}{p_n} = \frac{\sum_{k=0}^{\ell} \alpha_k (n)_k}{\sum_{k=0}^{m} \beta_k (n)_{k}}\,, \; \ell \leq j\,, \; m \leq j
\label{eq:prodgn}
\end{equation}
where $p_i$ is the first non-zero value of the positive sequence defining the pmf $p_j$, i.e. $\min_{i\geq0}p_i>0$ and hence satisfying $g(i)>0$ and $g(i-1)\leq 0$.
\end{Theorem}
\begin{Remark}
Class of discrete distributions on $\mathbb{Z}_{+}$ satisfying the recursion formula (\ref{eq:prodgn}) is substantially wider then usually discussed Kotz or Ord  family of distributions or their generalizations \cite{JKK2005,Chakraborty2015-ul}.
It is also worth nothing the similarity between Eq.~(\ref{eq:prodgn}) and pgfs satisfying Eq.~(\ref{eq:k1}). However, not every pgf satisfying Eq.~(\ref{eq:k1}) has its $g(n)$-function
so simple that it can be written as in Eq.~(\ref{eq:prodgn}). 

To illustrate this point let us give two examples which are related to Sibuya-like distributions. The first is the thinned version of the shifted Sibuya distribution (\ref{eq:sh_sbdpgf}) with the pgf $\mathcal{S}_0(\mathcal{B}_a(w))$. It is self-decomposable and so it can be expressed as the Laplace integral (\ref{eq:lapdens}). Using Eq.~ (\ref{eq:gnkl}) we can calculate  its $g$--function to obtain
\begin{equation}
\label{g_thshsbd}
g(n)=\frac{(n+1) (n+1-\gamma)} {(n+2)}
\cdot \frac{\, _2F_1(1,\gamma +1;n+3;1-a)}{\, _2F_1(1,\gamma +1;n+2;1-a)}    
\end{equation}
where $_2F_1(a_1,a_2;b_1;z)$ is the Gauss hypergeometrical function. Obviously only for $a=1$ the function $g(n)$ can be expressed as a ratio of two polynomials in $n$.

Another example is the pgf of the discrete stable distribution (\ref{eq:dsdpgf}). Although its function $f(x)$ does exist, see Eq.~\ref{eq:fdsd}, its $g(n)$ 
\begin{equation}
  g(n)=\frac{\sum_{i=0}^{n+1} (\lambda\gamma)^{i+1} F_i(\gamma)}{\sum_{i=0}^n (\lambda\gamma)^{i+1} F_i(\gamma)}\,,
\end{equation}
where the functions $F_i(\gamma)$ are some polynomials in $\gamma$
satisfying $F_i(1)=0, i>0$  with $F_0=1$, it does not allow to be written in the form of Eq.~(\ref{eq:prodgn}).
\end{Remark}

\begin{Corollary}
\label{hpg}
The pmf (\ref{eq:prodgn}) is the generalized hypergeometric probability distributions \cite{JKK2005}  with the pgf:
\begin{equation}
Q(w)\!=\! \frac{H(zw)}{H(z)}\,, 
\;H(z)= \sum_{n=0}^{\infty} \frac{z^n}{n!}  \prod_{i=j}^{n-1}g(i)= \,_{\ell}F_m(a,b,z)
\label{eq:gn2pgf}
\end{equation}
where $_{\ell}F_m(a,b,z)$ represents the generalized hypergeometric function with series expansion 
\begin{equation}
\label{eq:genhyp}
_{\ell}F_m(a_1, \ldots,a_{\ell}; b_1, \ldots, b_m;z)=
\sum_{n=0}^{\infty}\frac{(a_1)_n\ldots(a_{\ell})_n}
{(b_1)_n\ldots(b_m)_n} \frac{z^n}{n!}\,.   
\end{equation}
The constants $a_1,\ldots,a_{\ell}$ and $b_1,\ldots,b_m$ in
Eq.~(\ref{eq:genhyp}) are some  rational functions of $\alpha_1,\ldots,\alpha_{\ell}$ and $\beta_1,\ldots,\beta_m$ and $z=\alpha_{\ell}/\beta_m$ where $\alpha_{\ell}$ and $\beta_m$ are the last non-zero transition amplitudes. 
\end{Corollary}

\begin{proof}
Using Eq.~ (\ref{eq:prodgn})  together with the identity $n!=j!\prod_{i=j}^{n-1} (i+1)$  and $z=\alpha_{\ell}/\beta_{m}$ we obtain
\begin{equation}
\label{eq:ghf}
Q(w)\!=\! \!=\! p_j\sum_{n=0}^{\infty} w^n\prod_{i=j}^{n-1}\frac{g(i)}{i\!+\!1}\!=\!
p_j z^{-j}j! \sum_{n=0}^{\infty}\frac{(zw)^n}{n!} \prod_{i=j}^{n-1}
\frac{\sum_{k=0}^{\ell} \tfrac{\alpha_k}{\alpha_{\ell}} (i)_k}{\sum_{k=0}^{m} \tfrac{\beta_k}{\beta_m} (i)_{k}}=\frac{H(zw)}{H(z)}\,.
\end{equation}
The last equality with $1/H(z)= p_j z^{-j} j!$ follows from the condition $Q(1)=1$. Thus the pgf (\ref{eq:ghf}) corresponds to the power series distribution with
the function $H(z)$ which which is of the type of Eq.~(\ref{eq:genhyp}).
\end{proof}

\begin{Remark}
Note the certain arbitrariness when associating the coefficients $a_i$ and $b_i$ with $\alpha_i$ and $\beta_i$, respectively. This is due to the permutation symmetry of the generalized hypergeometric function (\ref{eq:genhyp}) w.r.t. its arguments $a_i$ and $b_i$ separately.
\end{Remark}

\begin{Example}
\label{ex:alpha_1}
Consider Eq.~ (\ref{eq:prodgn}) with $\ell=1,j=0$, i.e. the case with the all coefficients $\alpha_{0},\alpha_{1}, \beta_{0},\beta_{1}$ positive. Then
\begin{eqnarray}
\label{g01}
g(n)= \frac{\alpha_0\!+\!\alpha_1 n}{\beta_0\!+\!\beta_1 n}; \;\;\;
H\left(\frac{\alpha_1}{\beta_1}\right) = \sum_{n=0}^{\infty}\frac{(\tfrac{\alpha_1}{\beta_1})^n}{n!}\prod_{j=0}^{n-1}\frac{\frac{\alpha_0}{\alpha_1}\!+\! i}{\frac{\beta_0}{\beta_1}\!+\! i}\!=\! _1F_1\left(\frac{\alpha_0}{\alpha_1};\frac{\beta_0}{\beta_1}; \frac{\alpha_1}{\beta_1}\right)
\end{eqnarray}
where $_1F_1(a;b;z)$ is Kummer's (confluent hypergeometric) function.  

Neglecting in $g(n)$ of Eq.~(\ref{g01})  $\alpha_1$ and setting $\beta_0=\beta_1$ with $z=\alpha_0/\beta_1=\theta$ yields the pmf of  the Conway-Maxwell-Poisson (CMP) distribution with parameter $r=2$  \cite{Shm2005} 
\begin{equation}
\label{eq:cmp}
g(n)=  \frac{\theta}{n+1} ,~~ p_n =p_0\prod_{i=0}^{n-1}\frac{g(i)}{i+1}=p_0 \frac{\theta^n}{(n!)^2}\, ,~~
H(z)=\frac{1}{p_0}=\sum_{n=0}^{\infty}\frac{z^n}{(n!)^2}\,.
\end{equation}
Expression for the scaled factorial moments (\ref{eq:Fj}) of this pmf obtained in \cite{Daly2016-ec} reads
\begin{equation}
F_j=\frac{I_j(2\sqrt{\theta})}{(I_1(2\sqrt{\theta}))^j};\;\;
I_j(x)=\sum_{k=0}^{\infty}\frac{1}{k!\Gamma(j+k+1)}
\left(\frac{x}{2}\right)^{j+2k}
\end{equation}
It is worth mentioning that the pmf (\ref{eq:cmp}) has appeared in \cite{Ko:2000vp} as a stationary solution of the kinetic master equation describing the production of charged particles which are created or destroyed only in pairs due to the conservation of their charge.

On the other hand neglecting in $g(n)$ of Eq.~ (\ref{g01}) the constant $\beta_{1}$ we obtain the NBD
\begin{eqnarray}
\label{eq:nbdqk}
g(n)= \frac{\alpha_0+\alpha_1 n}{\beta_0} = q (k+n);\;\;\; p_n 
=p_0\frac{q^n (k)_n}{n!};\;\;\; q=\frac{\alpha_1}{\beta_0}\,, k=\frac{\alpha_0}{\alpha_1} \\ 
\sum_{n=0}^{\infty}p_n\!=\!1\!=\! p_0\left (1\!+\!\sum_{n=1}^{\infty}\frac{q^n (k)_n}{n!} \right)\!=\! (1-q)^{-k}; \;\;\;p_n=\binom{n+k-1}{k}(1-q)^k q^n. 
\label{eq:nbdpmf}
\end{eqnarray}
\end{Example}

\begin{Example}
\label{ex:alpha_2}
Consider now more general case of  Eq.~ (\ref{eq:prodgn}) with  $\ell=2,j=0$
\begin{equation}
\label{eq:g02}
g(n)= \dfrac{\alpha_0 + \alpha_1 n + \alpha_2 n(n - 1)}{\beta_0 +\beta_1 n};
\end{equation} 
With $z=\alpha_2/\beta_1$, $\alpha_0= \beta_1=\beta_0/2$ and $\alpha_1=3\alpha_2=\beta_0\theta/2$ we obtain the zero-inflated logarithmic distribution
\begin{equation}
\label{eq:zild}
g(n)= \frac{(n\!+\!1)^2}{n\!+\!2}\theta;\; \; \;  p_n =p_0\prod_{i=0}^{n-1}\frac{g(i)}{i+1}=-\ln(1\!-\!\theta)\dfrac{\theta^{n+1}}{n\!+\!1}.
\end{equation} 
On the other hand the logarithmic distribution can be reconstructed from    
\begin{equation}
g(n) = n\theta; \;\;\; p_n=p_1\prod_{i=1}^{n-1}\frac{g(i)}{i+1}; \;\;\; \frac{\alpha_1}{\beta_1} = \frac{\alpha_2}{\beta_1} = \theta, \alpha_0=\beta_0=0.
\label{Logg}
\end{equation}
And similarly to the above case for the geometric distribution with $g(n)=(n+1)\theta$ the only non-zero coefficients are $\alpha_1 = 2\beta_1\theta, \alpha_2 = \beta_1\theta$. 
\end{Example}

\begin{Example}
\label{ex:beta_2}
While the NBD (\ref{eq:nbdpmf}) is frequently used in different areas of science e.g. in medicine \cite{Carter2014-rs}, biology \cite{Lloyd-Smith2005-rx} or in particle physics \cite{Biyajima:1986yy,Giovannini:1985mz,Chliapnikov:1989ji}, the pmf obtained from Eq.~ (\ref{eq:prodgn}) with non-zero $\alpha_{0},\alpha_2,  \beta_{0},\beta_{2}$  and thus with the pgf Eq.~(\ref{eq:gn2pgf})  where
\begin{equation}
 H(z)=\,_2F_2(a_1, a_2;b_1, b_2;z); \; a_1+a_2=b_1+b_2=-1,\, a_1a_2=\frac{\alpha_0}{\alpha_2},\, b_1b_2 = \frac{\beta_0}{\beta_2},\, z=\frac{\alpha_2}{\beta_2}
\end{equation}
is seldom. Nevertheless, as shown in  \cite{Zborovsky:1996yj} it is important when describing the data with a secondary maximum (bump). 
\end{Example}

\section{Sibuya--like solutions of the stationary birth-death equations}
\subsection{Generalized Sibuya distribution}
The Sibuya distribution with the probability mass function
\begin{equation}
\label{sbd}
\mathbb P(N=n)=
\frac{\gamma(\gamma-1)\dots(\gamma-n+1)}{n!}
(-1)^{n+1}= \binom{\gamma}{n}(-1)^{n+1},\, n\in \mathbb N, 
\end{equation}
first appeared in \cite{art:Sibuya}. 
The generalized Sibuya distribution with parameters $\nu\geq 0$ and $0< \gamma<\nu +1$ has been introduced in \cite{Kozu2017} as the distribution of the number $N$ of trials until series of the Bernoulli trials with the varying success probability $\gamma/(\nu+n-1)$ and thus having the probability mass function
\begin{equation}
\label{gsbd}
p_n =\left(1-\frac{\gamma}{\nu +1}\right)\ldots \left(1-\frac{\gamma}{\nu +n - 1}\right)\frac{\gamma}{\nu +n }.
\end{equation} 
One connection of the generalized Sibuya distribution with the classical Sibuya variable $N$ is through the probability of the excess $\mathbb P(N-m=n|N>m)$, which has the generalized Sibuya distribution with $\nu=m$. This extends to the generalized Sibuya variable $N$ for which $\mathbb P(N-m=n|N>m)$ has also the generalized Sibuya distribution with $\nu+m$ as its parameter.

 Writing  the pmf (\ref{gsbd}) as 
\begin{equation}
\label{gsbd1}
 p_n=p_1\prod_{i=1}^{n-1}\frac{g(i)}{i+1};\; \; \; p_1=\frac{\gamma}{\nu +1 };\; \; \;  
 g(i)= \frac{(i+1)(i+\nu-\gamma)}{i+\nu +1}.   
\end{equation}  
we can see that function $g(i)$ of the generalized Sibuya distribution is a special case of Eq.~(\ref{eq:g02}) with $\beta_2=0$ and thus solves  the stationary B-D equations (\ref{Fn}). The corresponding non-zero coefficients are 
\begin{equation}
\label{eq:abgensbd}
    \frac{\alpha_2}{\beta_1}=1,\;\; \frac{\alpha_1}{\beta_1}=2+\nu-\gamma,\;\; \frac{\alpha_0}{\beta_1}=\nu-\gamma,\;\; \frac{\beta_0}{\beta_1}=\nu+1\,.
\end{equation}
Moreover for $\nu=0$ we have $p_1=\gamma$ and Eq.~ (\ref{gsbd1})  reduces to the pmf of  Sibuya distribution (\ref{eq:sbdpgf}).

The pgf of the generalized Sibuya distribution calculated from Eq.~ (\ref{gsbd1}) reads 
\begin{equation}
\label{pgfgsbd}
Q(w)  =\frac{w \gamma}{\nu + 1}\cdot\,
_2F_1(1, \nu-\gamma + 1;  \nu+ 2; w)\,,
\end{equation}
where $_2F_1(a_1,a_2;b_1;z)$ is the Gauss hypergeometric function.

\subsection{Extended Sibuya distribution}
Consider now Eq.~(\ref{eq:prodgn}) with $\ell=2,j=1$ and hence with the function $g(n)$ given by Eq.~(\ref{eq:g02}) with $\alpha_0=\beta_0=0$
\begin{eqnarray}
\label{g12}
\; \; \; g(n)=\frac{\alpha_2 n(n\!-\!1)+\alpha_1 n}{\beta_1 n}=b(n-\gamma); \;\;\;\; \frac{\alpha_2}{\beta_1}=b,\; \frac{\alpha_1}{\beta_1}=b(1-\gamma),\; \gamma=1-\frac{\alpha_1}{\alpha_2}.
\end{eqnarray}
Using  Corollary \ref{hpg} it is easy to verify that the corresponding pgf 
\begin{equation}
\label{eq:prsbd}
 \mathcal{R}(b,\gamma,w) = \frac{H(b w)}{H(b)} = 
\frac{1 - (1- bw)^{\gamma}}{1 - (1 -b)^{\gamma}} = \frac{\mathcal{S}(\gamma, b w)}{\mathcal{S}(\gamma, b)}
\end{equation}
represents yet another generalization of the Sibuya pgf $\mathcal{S}(\gamma, w)$. For $\gamma = -k <0$ and fixed $b=q$ the pgf (\ref{eq:prsbd}) coincides with the NBD pgf  (\ref{eq:nbdpgf}) truncated at zero
\begin{equation}
\label{eq:nbdtrunc}
\mathcal{R}(w)   =  \frac{1 - (1- qw)^{-k}}{1 - (1-q)^{-k}} = 
\frac{Q_{NBD}(w) -Q_{NBD}(0)}{1-Q_{NBD}(0)};~~~Q_{NBD}(w)=\left(\frac{1-q}{1-qw}\right)^k\, .
\end{equation}
For $\gamma\to 0$, i.e. when $\alpha_2\to\alpha_1$, the pgf (\ref{eq:prsbd}) approaches the logarithmic distribution
\begin{equation}
\lim_{\gamma \to 0}\frac{1 -(1- bw)^{\gamma}}{1- (1 -b)^{\gamma}} = \frac{\log(1-b w)}{\log(1-b)}\, 
 \label{logpgf}   
\end{equation}
and for $\gamma \to 1$  it converges to distribution with the pmf $p_n=\delta_{1n}$. 

Let us note that the pgf (\ref{eq:prsbd}) was already mentioned in \cite{Letac2019-cm} (see Eq.(6) therein) as the natural exponential family extension of the Sibuya distribution. Since no name was given to it nor its relation with the other known distributions was discussed  we take a liberty and call it the pgf of {\it extended}  Sibuya distribution stressing its smooth transition in $\gamma$ to the pgfs of several fundamental discrete distributions. 

Let us now study the factorial moments $\E X^{(j)} = \langle n^{(j)} \rangle :=\langle n(n-1)\ldots (n-j+1) \rangle, ~j\geq1$, see Eq.(\ref{eq:Fj}), of the pgf $\mathcal{R}(w)$
\begin{equation}
 \langle n^{(j)} \rangle = \mathcal{R}^{(j)}(1)=\frac{(-b)^j (1-b)^{\gamma -j} \gamma ^{(j)}}{(1-b)^{\gamma }-1}\, ; ~~\gamma^{(j)}=\frac{\Gamma(\gamma\!+\!1)}{\Gamma(\gamma\!+\!1\!-\!j)}\, , ~ ~j\geq1.
\end{equation}
When approaching the Sibuya distribution, i.e. for $b\to 1^{-}$ we obtain expected results
\begin{equation}
  \langle n \rangle = \mathcal{R}^{(1)}(1) =
  \frac{p_1} {(1\!-\!b)^{1-\gamma}} \xrightarrow[b \to 1^{-}]{} \infty; \;\; p_1 = \mathcal{R}^{(1)}(0) = \frac{ b \gamma}{1-(1-b)^{\gamma}}~ \xrightarrow[b \to 1^{-}]{} ~\gamma. 
\end{equation}
For the second scaled factorial moment $F_2$, cf. Eq.~ (\ref{eq:Fj}) we have
\begin{equation}
F_2=
\frac{\langle n(n-1) \rangle}{\langle n \rangle ^2} =   
\frac{(1\! -\! \gamma) \left(1 \! - \!(1-b)^{\gamma }\right)}{\gamma (1-b)^{\gamma }}=\frac{1\! -\! \gamma}{\gamma} \cdot \left(e^{\delta\gamma}\!-\!1\right); \;\;\delta:=-\log(1\!-b).
\label{eq:F2R}
\end{equation}
Ignoring the region with $\gamma <0$ which corresponds to the zero-truncated NBD let us analyze the behavior of  
\begin{equation}
\label{F2d1}
 \frac{dF_2}{d\gamma} = \frac{e^{ \gamma \delta} (\gamma\delta (1\!-\!\gamma) \!-\!1)\!+\!1}{\gamma ^2}\,.
 \end{equation}
For $\gamma\to 0$, i.e. for  the logarithmic distribution, $F_2$ has a minimum at $\delta=2$. Further up, in the region  $(\gamma> 0, \delta>2)$, the extreme of $F_2(\gamma)$ moves along the trajectory $dF_2/d\gamma|_{\gamma_0}=0$ or equivalently $\psi(\gamma_0,\delta)=0$, where $\psi(\gamma,\delta)= (e^{\gamma\delta}(\gamma\delta(1-\gamma)-1)+1$. While $F_2(\gamma)$ increases for $0<\gamma<\gamma_0$  it decreases for $\gamma_0<\gamma<1$. 
When approaching the singularity at $b=1$  the second scaled factorial moment  $F_2$ and hence also its maximum goes to infinity.
Let us add that the same maxima appear also in the higher scaled factorial moments $F_j,j>2$.

\subsection{The shifted variants of extended and generalized Sibuya distributions}
Consider random variable $Y\preq X+1$ where $Y\in \mathbb{N}$ has either extended or generalized Sibuya distribution.
Then the rv $X \in \mathbb{Z}_{+}$ has {\it shifted extended or generalized Sibuya distribution}, respectively. The case when $Y$ has ordinary Sibuya distribution was already discussed in
Example \ref{ex:shsbd}. Note that contrary to the original extended or generalized Sibuya distributions their shifted variants are now defined on $\mathbb{Z}_+$
and so they may be self-decomposable.

\subsubsection{The shifted extended Sibuya distribution}

If $p_n$ and $\mathcal{R}(b,\gamma,w)$ are the pmf  the pgf of the extended Sibuya distribution (\ref{g12}), respectively, then $r_{n}=p_{n+1}$  and $\mathcal{R}_0(b,\gamma,w)=\mathcal{R}(b,\gamma,w)/w$, cf. Eq.~(\ref{eq:sh_sbdpgf}), are the pmf and the pgf of the shifted extended Sibuya distribution. Corresponding  function $g(n)$ thus reads
\begin{equation}
\label{gshftsbd}
  g_r(n) = \frac{(n+1) r_{n+1}}{r_{n}}=\frac{n+1}{n+2}\cdot\frac{(n+2)p_{n+2}}{p_{n+1}} =
  \frac{n+1}{n+2}b(n+1-\gamma); \;\; b\leq 1\,
\end{equation}
cf. Eq.~(\ref{g_shsbd}).
The elementary transition amplitudes (\ref{eq:prodgn}) derived from Eq.~(\ref{gshftsbd}) are
\begin{equation}
\label{abshexysbd}
    \frac{\alpha_2}{\beta_1}=b,\;\; \frac{\alpha_1}{\beta_1}=b(3-\gamma),\;\; \frac{\alpha_0}{\beta_1}=b(1-\gamma),\;\; \frac{\beta_0}{\beta_1}=2\,.
\end{equation}

\begin{Theorem}
The shifted extended Sibuya distribution is self-decomposable.
\end{Theorem}
\begin{proof}
 Following \cite{Christoph2000} we use 
\begin{Lemma} {\rm \cite{Bondesson1992-dx}}
\label{Bond1}
A strictly decreasing pmf $r_n,n=0,1,2,\dots$ such that
\begin{equation}
\label{bondel}
\max_{0\leq n\leq j}\frac{r_{n+1}}{r_n} \leq
\frac{j+2}{j+1}\cdot \frac{r_{j+1}-r_{j+2}}{r_{j}-r_{j+1}}\, , \; \; \forall j=0,1,2,\ldots
\end{equation}
is discrete self-decomposable.
\end{Lemma}
 In our case the sequence 
$$\frac{r_{n+1}}{r_n}=\frac{p_{n+2}}{p_{n+1}}=b\left(1-\frac{1+\gamma}{n+2}\right)
$$
 is strictly decreasing and therefore the left hand side of Eq.~(\ref{bondel}) gives:
 $$ \max_{0\leq n\leq j}\frac{r_{n+1}}{r_n} =
 b\left(1-\frac{1+\gamma}{j+2}\right).
 $$
For the right hand side of the inequality (\ref{bondel}) we have
$$ \frac{j+2}{j+1}\cdot \frac{r_{j+1}-r_{j+2}}{r_{j}-r_{j+1}}=
\frac{(j+2)}{(j+1)}\cdot \frac{b(j+1-\gamma)(j+3-b(j+2-\gamma))}{(j+3)(j+2-b(j+1-\gamma))}.
$$
Dividing now the r.h.s. of the Eq.~(\ref{bondel}) by its l.h.s. we obtain 
$$R(b,j)=\frac{j+2}{j+1}\cdot \frac{r_{j+1}-r_{j+2}}{r_{j}-r_{j+1}}\Big /\max_{0\leq n\leq j}\frac{r_{n+1}}{r_n}=\frac{(j+2)^2 (j+3-b(j+2-\gamma))}{(j+1)(j+3)(j+2 -b(j+1-\gamma))}.$$
Now we need to prove that $R(b,j)\geq 1$. It is easy to verify that $R(b,j)$ is decreasing function with respect to $b$ for each fixed $j$. Therefore, its minimum is attained at point $b=1$ where
$$ \frac{j^2+4j+4}{j^2+4j+2}>1. $$
Hence Eq.~(\ref{bondel}) is satisfied and the shifted extended Sibuya distribution is indeed
self-decomposable and thus also infinitely divisible \cite{StvH2004}.
\end{proof}

For $b<1$ the Theorem \ref{eq:binrand} tells us that the scaled factorial moments of the shifted extended Sibuya distribution with the pgf  $\mathcal{R}_0(b,\gamma,w)$ and of its thinned version $\mathcal{R}_0(b,\gamma,1-a+a w) $ are  the same. Moreover, since $\mathcal{R}_0(b,\gamma,w)=\mathcal{S}_0(\gamma,bw)/\mathcal{S}_0(\gamma,b)$, see Eq.~(\ref{eq:prsbd}),  it is easy to check using Eq.~\ref{eq:lapdensa} with $f(x)$ given by Eq.~ (\ref{eq:sh_sbdfx}) that the pgf of the shifted extended Sibuya distribution satisfies the Theorem \ref{thK} and hence the scaling variable $a$ can be any positive number. 

\subsubsection{The shifted generalized Sibuya distribution}
If $p_n$ is the pmf of the generalized Sibuya distribution (\ref{gsbd}) then $r_{n}=p_{n+1}$ is the pmf of the shifted generalized Sibuya distribution with the $g$-function  
\begin{equation}
\label{gshfgsbd}
  g_r(n)=\frac{(n+1)r_{n+1}}{r_n} =\frac{n+1}{n+2}\cdot\frac{(n+2)p_{n+2}}{p_{n+1}} =
  \frac{(n+1)(n+1+\nu-\gamma)}{n+2+\nu}
\end{equation}
and hence
\begin{equation}
\label{abshgensbd}
    \frac{\alpha_2}{\beta_1}=1,\;\; \frac{\alpha_1}{\beta_1}=(3+\nu-\gamma),\;\; \frac{\alpha_0}{\beta_1}=(1+\nu-\gamma),\;\; \frac{\beta_0}{\beta_1}=\nu+2\,.
\end{equation}
Interestingly, the only difference between
the coefficients of the original (\ref{eq:abgensbd}) and the shifted distribution (\ref{abshgensbd}) consists in replacement $\nu \to \nu+1$.

To check the self-decomposability we once again use the Lemma \ref{Bond1} to find that
$$\frac{r_{n+1}}{r_n}=\frac{p_{n+2}}{p_{n+1}}=\frac{n+1+\nu-\gamma}{n+2+\nu}=1-\frac{1+\gamma}{n+2+\nu}
$$
is strictly decreasing. Therefore the l.h.s. of the Eq.~(\ref{bondel}) gives:
 $$ M=\max_{0\leq n\leq j}\frac{r_{n+1}}{r_n} =
 \left(1-\frac{1+\gamma}{j+2+\nu}\right).
 $$
Using Eq.~(\ref{gshfgsbd}) we rewrite the r.h.s. of the Eq.~(\ref{bondel}) as
$$ N=\frac{j+2}{j+1}\cdot\frac{r_{j+1}-r_{j+2}}{r_{j}-r_{j+1}} = \frac{g_r(j)}{j+1}\cdot \frac{j+2-g_r(j+1)}{j+1-g_r(j)} =-\frac{(j+2) (\gamma +2 \nu -1)^2}{(j+1) (j-\nu +2) (j-\nu +3)}$$ 

Dividing now the r.h.s. of the Eq.~(\ref{bondel}) by its l.h.s. and neglecting terms $O\left(j^{-3}\right)$
we observe that there always $\exists j_0>0$ such that $\forall j\geq j_0$  the ratio

$$\frac{N}{M}\approx \frac{(\gamma +2 \nu -1)^2}{j^2} <\frac{(3\nu)^2}{j^2}<1$$ 
where in that last equation we have used the inequality $0< \gamma<\nu +1$. Thus contrary to the previous case the shifted generalized Sibuya distribution is not self-decomposable.

\subsection{Beyond the Sibuya distribution}
Let us consider the following four-parameter generalization of the Sibuya pgf 
\begin{equation}
\label{genkl}
Q(w)=\frac{1-(1-(1-(1-bw)^{\gamma})^{\ell})^k}
{1-(1-(1-(1-b)^{\gamma})^{\ell})^k} =
\frac{1-(1-(H(b,\gamma,w))^\ell)^k}{1-(1-(H(b,\gamma,1))^\ell)^k}
. 
\end{equation} 
In \cite{Kl2021} it was shown that Eq.~(\ref{genkl}) represents a probability generating function in the case of $\gamma =1/m$, $k=m$, $b=1$ where $m, \; \ell \in \mathbb N$. It is not difficult to prove this statement remains true for the case of positive integer $k \leq m$ and $b\in (0,1]$. For the case of $k=\ell=1$ the function $Q(w)$ coincides with probability generating function $\mathcal{R}(w)$ of the extended Sibuya distribution (\ref{eq:prsbd}). For the case of $k=1$  the pgf Eq. (\ref{genkl}) represents  the sum of $\ell$ extended Sibuya-- distributed random variables. In particular, for $\ell=2$, the function $g(n)$ is easily calculable
\begin{equation}
\label{ell2}
    g(n) = \frac{b(n+1)\Bigl( \binom{2 \gamma}{n+1} -2 \binom{\gamma}{n+1}\Bigr)}{2 \binom{\gamma}{n} - \binom{2\gamma}{n}}\, .
\end{equation}
Let us note that for $\gamma\geq 1/2$ the function (\ref{ell2}) is positive and for $n\geq2$ it is close to linear function of $n$  $g(n)\approx a_1(\gamma)(n+a_2(\gamma))$. In particular for $\gamma=1/2$ it coincides with the extended Sibuya distribution (\ref{g12}).

On the other hand for $k=\ell=2$, $m\geq 2$ the ratio $g(n)$ (\ref{eq:prodgn})  takes the form 
\begin{equation}
g(n) =
-b \frac{(1+n) \left(4 \binom{2/m}{1+n}-4 \binom{3/m}{1+n} +\binom{4/m}{1+n}\right)}{2
\left(4 \binom{2/m}{n}-4 \binom{3/m}{n} +\binom{4/m}{n}\right)}.
\end{equation}
Most simple case is $m=2$ with $g(n) = \frac{b}{4}(2n-3)$ which is positive for $n\geq2$. The pmf then reads
\begin{equation}
p_n=p_2\prod_{i=2}^{n-1}\frac{b}{4}\frac{2i-3}{i+1}
=\frac{\alpha_3 n(n\!-\!1)(n\!-\!2)+\alpha_2n(n\!-\!1)}{\beta_2 n(n\!-\!1) }; \;\; \frac{\alpha_3}{\beta_2}=\frac{b}{2},\; \frac{\alpha_2}{\beta_2}=\frac{7b}{4}.
\end{equation} 

\subsection{The extended Sibuya as  progeny}
Consider branching process $(Z_n)_{n\geq 0}^{\infty}$  governed by the pmf $p_k, k\in\mathbb{Z}_+$ \cite{Feller1, Harris2012-dj}.  Markov property of the process determines how the direct descendants of the $n$th generation form the $(n\!+\!1)$st generation. 
If $Z_n$ is the size of the $n$th generation  and $H_n(w)$ is its pgf then rv $Y_n= 1+ Z_1 + \ldots + Z_n$ with the pgf $Q_n(w)$ represents the total number of descendants up to the $n$th generation  including the ancestor (zero generation). By  assumption $Z_0=1$ and  $H_1(w)= H(w)=\sum_{k=0}^{\infty}p_k w^k$ and therefore
\begin{equation}
\label{Qn}
 Q_1(w) = wH(w); ~~~  Q_n(w) = wH ( Q_{n-1}(w)). 
\end{equation}
If $H^{(1)}(1)=\sum_0^{\infty}k p_k \leq 1$ then the rv $Y=\sum_{n=0}^{\infty}Z_n$  is finite and is called progeny \cite{Feller1}. The pgf of  the progeny  then corresponds to a fixed point of Eq.~(\ref{Qn})
\begin{equation}
\label{fxtpt}
  \lim_{n\to\infty}Q_n(w) = Q(w) = wH(Q(w)).
\end{equation}
By inverting the last equation we can express  $H(w)$ in terms of  $Q(w)$
\begin{equation}
\label{progenH}
Q(w)=w H(Q(w));~~ w= Q^{-1}(u) =g(u); \;\; H(u)=\frac{u}{g(u)}.
\end{equation}
Note that although such inversion is in principle always possible the rv $Y$ is the progeny only when all the expansion coefficients of the function $H(u)$  in powers of $u$  are positive. 
In particular, for the case of Sibuya pgf (\ref{eq:sbdpgf}) Eq.~(\ref{progenH}) with $Q(w)=\mathcal{S}(\gamma, w)$ the progeny was shown to exists only for $\frac{1}{2}\leq\gamma<1$ \cite{Letac2019-cm}. 

For the extended Sibuya distribution Eq.~(\ref{progenH}) with $Q(w)=\mathcal{R}(w)$ yields
\begin{eqnarray}
\label{Hu}
     ~~H(u,b,\gamma)=\frac{u b}{1-(1-u c)^{1/\gamma}}; ~ c=1-(1-b)^{\gamma};~~ \gamma \neq 0.
    \end{eqnarray}
Expanding $H(u, b,\gamma)$ in powers of $u$
\begin{equation}
\label{Huexp}
H(u,b,\gamma)= \frac{b \gamma }{c}
-\frac{1}{2} b (\gamma -1) u
-\frac{\left(b c \left(\gamma ^2-1\right)\right)}{12 \gamma }u^2 
-\frac{\left(b c^2 \left(\gamma ^2-1\right)\right)}{24 \gamma }u^3  +O\left(u^4\right).   
\end{equation}
we observe that for $\gamma<-1$ the coefficient $\sim -(\gamma ^2-1)u^2$ becomes negative (taking into account that $b c/\gamma>0$) and for $-1< \gamma<0$ the same is true for the coefficient $\sim -(b c^2 \left(\gamma ^2-1\right))u^3$. 
On the other hand for $\gamma=-1$ with $H(u)=-\frac{b}{c}+b u$ the corresponding Geometric distribution pgf $\mathcal{R}(w,b,-1)$  is progeny.

Further on at $\gamma=0$ 
\begin{equation}
 H(u,b,0)=\frac{bu}{1-(1-b)^u} =
 -\frac{b}{\log (1-b)}+\frac{b u}{2}-\frac{1}{12} u^2 (b \log (1-b))+\frac{1}{720} b u^4 \log ^3(1-b)+O\left(u^5\right)
\end{equation}
it is now coefficient $\sim  \log ^3(1-b) u^4 $ which is negative. 
Finally for  $\gamma>0$  substitution $t=bw\to w$ in Eq.~(\ref{eq:prsbd}) yields the pgf of Sibuya distribution (\ref{eq:sbdpgf}) $\mathcal{S}(t)$ multiplied by positive factor  $1/(1-(1-b)^{\gamma})$ and hence with the identical signs of all expansion coefficients. Therefore the results obtained in \cite{Letac2019-cm} for  $\gamma>0$  are unchanged and extended Sibuya pgf is progeny for $\gamma \in \langle\tfrac{1}{2},1)$ while for $0<\gamma<\tfrac{1}{2}$ it is not. 
In terms of first and second (factorial) moments of the pgf (\ref{Hu})
 \begin{eqnarray}
 \langle k \rangle = H^{(1)}(1,b,\gamma)\!=\!1\!+\!\frac{(1\!-\!b)c}{b\gamma} \leq 1 \\
\langle k(k-1) \rangle = H^{(2)}(1,b,\gamma)\!=\!
 \frac{c (1-b)^{1-2 \gamma } \left[2 (1\!-\!b) (c - b (\gamma +1) (c- \gamma )\right]}{b^2 \gamma ^2}.
 \end{eqnarray}

While in the regions $-1\gamma<0 $ and $0<\gamma<\frac{1}{2}$ the function  $H^{(2)}(1,b,\gamma)$  has non-monotonic dependence on the parameter $b$ with a maximum at $\frac{3}{4}<b_{max}<1$ rapidly moving towards 1, for $\gamma\geq \frac{1}{2}$ it becomes a monotonously increasing function of $b$. The change in the vicinity of $\gamma=\frac{1}{2}$ is quite dramatic:
\begin{equation}
    \lim_{b\to1^-}H^{(2)}(1,b,\gamma < \tfrac{1}{2})= 0;~~\lim_{b\to1^-}H^{(2)}(1,b,\gamma \geq \tfrac{1}{2})= \infty.
\end{equation}

\section{Conclusion}
The goal of this article was to discuss the properties of Sibuya-like random variables related to their scaling and birth-death processes. We have shown that their pmf satisfy a one-step recurrence relation $g(n)=(n+1)p_{n+1}/p_n$ with some of them 
representing the stationary solutions of the birth-death equations and most of them being self-decomposable.
In addition to already known Sibuya-like distribution (shifted Sibuya, generalized Sibuya, disctere stable) we have found  a new one - extended Sibuya distribution. The latter is very similar to ordinary Sibuya distribution having the same bounds of its validity as a progeny in branching process. 

\section*{Acknowledgements}
  The work of Lev B. Klebanov was partially supported by Grant 19-28231X GA \v{C}R.  Michal \v{S}umbera is partially supported by the grants LTT17018 and LTT18002 of the Ministry of Education of the Czech Republic.

\bibliographystyle{authordate1}
\bibliography{main}
\end{document}